\documentclass[12pt, reqno]{amsart}
\usepackage{amsmath,amsfonts,amssymb,amsthm,enumerate,appendix,caption}
\usepackage{cmap,mathtools}
\usepackage{paralist,bm}
\usepackage{graphics} %% add this and next lines if pictures should be in esp format
\usepackage{epsfig} %For pictures: screened artwork should be set up with an 85 or 100 line screen
\usepackage{graphicx}
\usepackage{braket}
\usepackage{epstopdf}%This is to transfer .eps figure to .pdf figure; please compile your paper using PDFLeTex or PDFTeXify.
 \usepackage[colorlinks=true]{hyperref}
\usepackage{xcolor}
\definecolor{myblue}{rgb}{0.0, 0.0, 1.0}
\definecolor{mygreen}{rgb}{0.01,0.75,0.20}
\hypersetup{ linkcolor=myblue, colorlinks=true, urlcolor=black, citecolor=red}
\usepackage{hyperref}
\newtheorem{theorem}{Theorem}[section]
\newtheorem{corollary}[theorem]{Corollary}

\newtheorem{lem}[theorem]{Lemma}
\newtheorem{lemma}[theorem]{Lemma}
\newtheorem{proposition}[theorem]{Proposition}

\theoremstyle{definition}
\newtheorem{definition}[theorem]{Definition}

\newtheorem*{remark}{Remark}
\newtheorem*{rem}{Remark}

\newtheorem{example}[theorem]{Example}
\theoremstyle{definition}
\usepackage[margin=1in]{geometry}

\numberwithin{equation}{section}

\def\ga{\alpha}

            \def\gl{\lambda}
                 
    \def\gr{\rho}

     \def\Gd{\Delta}

\newcommand{\eat}[1]{}

\DeclarePairedDelimiter\norm{\lVert}{\rVert}%
\makeatletter
\let\oldnorm\norm
\def\norm{\@ifstar{\oldnorm}{\oldnorm*}}

\newcommand{\Om} {\Omega}

\newcommand{\ra} {\rightarrow}

\newcommand\restr[2]{{% we make the whole thing an ordinary symbol
  \left.\kern-\nulldelimiterspace % automatically resize the bar with \right
  #1 % the function
  \right|_{#2} % this is the delimiter
  }}

\def\w{{\widetilde w}}

\def\w2{{W^{1,2}_0(\Om)}}
\def\hh2{{H^1_0(\Om)}}

\def\C{{\mathcal C}}
\def\D{{\mathcal D}}
\def\E{{\mathcal E}}

\def\N{{\mathbb N}}
\def\F{{\mathcal F}}

\def\cp{{\rm Cap}_{\sol{u}}}

\def\R{{\mathbb R}}

\def\({{\Big(}}
\def\){{\Big)}}

\def\ws2{{\F_{\frac{N}{2}}}}

\def\c1{{\C_c^1}}

\def\H{{\mathcal{H}}}

\newcommand\sol[1]{{{#1}}}
\newcommand\Q{Q}
\newcommand\QQ{Q'}
\renewcommand{\L}{\Delta_{p}}
\newcommand{\ph}{\varphi}
\renewcommand{\H}{\mathcal{H}}

\renewcommand{\S}[1]{S_{{\sol{u}}}(#1)}
\renewcommand{\E}{E_{{\sol{u}}}}

\newcommand{\Hmm}[1]{\leavevmode{\marginpar{\tiny%
			$\hbox to 0mm{\hspace*{-0.5mm}$\leftarrow$\hss}%
			\vcenter{\vrule depth 0.1mm height 0.1mm width \the\marginparwidth}%
			\hbox to
			0mm{\hss$\rightarrow$\hspace*{-0.5mm}}$\\\relax\raggedright #1}}}
\begin{document}
	\title[The space of Hardy-weights]{The space of Hardy-weights for quasilinear operators on discrete graphs}
	
	%\iffalse
	\author {Ujjal Das}
	
	\address {Ujjal Das, Department of Mathematics, Technion - Israel Institute of
		Technology,   Haifa, Israel}
	
	\email {ujjaldas@campus.technion.ac.il, getujjaldas@gmail.com}
\author[M.~Keller]{Matthias Keller}
\address{M.~Keller,  Institut f\"ur Mathematik, Universit\"at Potsdam,
	14476  Potsdam, Germany}
\email{matthias.keller@uni-potsdam.de}	

	\author{Yehuda Pinchover}
	\address{Yehuda Pinchover,
		Department of Mathematics, Technion - Israel Institute of
		Technology,   Haifa, Israel}
	\email{pincho@technion.ac.il}
	
	%\fi
	%%%%%%%%%%%%%%	
	\begin{abstract}
We study Hardy inequalities for $ p $-Schr\"odinger operators on general weighted graphs. Specifically, we prove a Maz'ya-type result, where we characterize the space of Hardy weights for $ p $-Schr\"odinger operators via a generalized capacity. The novel ingredient in the proof is the demonstration that the simplified energy of the $ p $-Schr\"odinger energy functional is compatible with certain normal contractions. As a consequence, we obtain a necessary integrability criterion for Hardy weights. Finally, using some tools of criticality theory, we investigate the existence of minimizers in the Hardy inequalities and discuss relations to Cheeger type estimates.

		\medskip
		
		\noindent  2020  \! {\em Mathematics  Subject  Classification.}
		Primary  \! 39A12; Secondary 35R02,  49J40, 31C20, 35J62.\\[1mm]
	\end{abstract}

\maketitle

\section{Introduction}
In this article we study quasilinear Schr\"odinger type operators on discrete graphs. Specifically, we are interested in the space of Hardy weights for these operator. While Hardy formulated his original inequality in the discrete setting of the natural numbers, it has been mainly studied in the continuum setting. There Hardy inequalities are a major tool in partial differential equations, operator theory and mathematical physics. 

In the last decade a strong interest in Hardy inequalities  on graphs arose. 
There is a mulitude of work on Hardy-type inequalities for one-dimensional operators  \cite{FKP_N,GKS, Gupta1,Gupta2,HuangYe,KPP_N,KS,Kostenko,Lefevre} including the fractional Laplacian \cite{CiaurriRoncal,KellerNietschmann} and on the Euclidean lattice   \cite{Gupta3,Gupta4,KapitanskiLaptev,KellerLemm}. Furthermore, it has been investigated on trees \cite{BSV,EFK,Golenia}. For general graphs {\it{optimal Hardy weights}} have been studied in 
\cite{KPP_optimal} for the linear case $ p=2 $ with  applications in \cite{KPP_rellich,KP_Agmon}, for further considerations  see \cite{BGK,FKMN,MurmannSchmidt}. For general $ p $, the investigations on optimal Hardy weights  have been extended in \cite{Florian_optimal}, for recent developments on more general settings beyond graphs see  \cite{CaoGrigoryanLiu,FrankSeiringer,Schmidt, Takeda1,Takeda2}.

Here, we consider a quasilinear Schr\"odinger type operator of the form \eqref{e:Q'}, which is the sum of a weighted $ p $-Laplacian and a possibly sign-changing potential. Throughout the present article we refer to this operator as  a $p$-Schr\"odinger operator.

Our first aim is to characterize the space of Hardy weights for $p$-Schr\"odinger operators on discrete graphs. In the continuum case, the problem of characterizing the Hardy weights goes back to the works of Talenti \cite{Talenti}, Tomaselli \cite{Tomaselli}, Muckenhoupt \cite{Muckenhoupt}, where the one dimensional case is studied; see also the monograph  \cite{OK} by Opic and Kufner for a detailed study on the ODE case. In higher dimension, there is a characterization for radial Hardy weights using the Bessel pairs \cite{LamLu}. By means of  the $ p $-capacity, Maz'ya \cite{Mazya2} provided a necessary and sufficient condition for the Hardy weights of $ p $-Laplacian on Euclidean domains. As a particular case, this characterization identifies the admissible domains for the Poincar\'e inequality, and also, for the Hardy inequality with distance function \cite{Kinnunen}.  Maz'ya's characterization also helps us to study the {\it{quasi-additivity property of the $p$-capacity}} \cite{Lehrback}.

 Using a generalized $ p $-capacity, Maz'ya's characterization  was recetly extended in \cite{Das_Pinchover1} to quasilinear Schr\"odinger operators in the continuum, where the novelty  was to overcome the difficulty due to the sign-changing behaviour of the potential.  This was realized by the so-called {\it{simplified energy functional}} \cite{ptt}. This study has been further extended to the Finsler settings in the continuum \cite{Hou}.  Of course, the major challenge in studying graphs comes from the non-local behavior of the discrete $ p $-Laplacian as it also occurs for the fractional $ p $-Laplacian in the continuum case. The Maz'ya type characterization for the latter case was achieved in \cite{Dyda}. Furthermore, a simplified energy for discrete $ p $-Schr\"odinger operators was recently developed by Fischer \cite{Florian_nonlocal}. We combine these ideas to obtain a characterization of Hardy weights for discrete $ p $-Schr\"odinger operators  (Theorem~\ref{Thm:Char}) on arbitrary graphs which may not be locally finite. A crucial novel  ingredient is to show that the simplified energy is compatible with suitable normal contractions which might be applicable also in other situations, see Lemma~\ref{lem:contr}, Lemma~\ref{l:cuttoff} and the proof of Lemma~\ref{Lem:ujjalp>2}. As we allow the graph to be non-locally finite, the characterization of Hardy weights given in Theorem~\ref{Thm:Char} is also valid for discrete fractional $ p $-Schr\"odinger operators (see Example \ref{Exam:frac}). Based on our characterization, we provide a simple necessary integrability condition on the space of Hardy weights in terms of a positive solution of minimal growth at infinity in Theorem~\ref{Thm:KP}. Such a result is found in the continuum setting in \cite{KP} and \cite{Das_Pinchover1}.

Our second aim is to provide a sufficient condition for the existence of a minimizer for a given Hardy weight with the best constant. This is achieved  under a spectral gap condition, Theorem~\ref{Thm:best_constant}. In the contiuum such considerations go back to \cite{MMP}. It is notable that such question of the existence of minimiser in the contiuum case is mainly studied using the concentration compactness \cite{DA,Smets,Tertikas}. On the other hand, in \cite{MMP,Lamberti_Pinchover} and recently in \cite{Das_Pinchover1,DDP, Hou}, it was shown that such problems can be attacked using criticality theory. Our proof of Theorem~\ref{Thm:best_constant} relies on criticality theory for discrete graphs. Finally, based on the observations that the best Hardy constant can be interpreted as a spectral quantity, we estimate  in Section~\ref{sec_Cheeger} the Hardy  constant by suitable Cheeger constants.

The article is structured as follows. In Section~\ref{sec-prelim} we introduce the relevant notions such as the $p$ Schr\"odinger operators together with their energy functional, the simplified energy functional, and the generalized $p$-capacity.  In Section~\ref{sec_Mazya} we prove the characterization of Hardy weights in terms of the generalized $p$-capacity, and in Section~\ref{KP-condition} we provide for Hardy weights a necessary integrability criterion. In Section~\ref{sec_existence} we study the existence of minimizers and in Section~\ref{sec_Cheeger} we consider the relation to Cheeger constants.

\section{Preliminaries}\label{sec-prelim}

Let $ p\in (1,\infty) $ and  $ X $ be an infinite countable set equipped with the discrete topology. 
 Denote $C(X)={\{f:X \rightarrow \R \}}$ and
$C_c(X)=\{f\in C(X) \mid {\rm{supp}}(f) \Subset X\}$, where  $ K\Subset X$ means that $ K $ is a compact, i.e. finite, subset of $ X $.
For a positive or an absolutely summable function $ f $ on a discrete set $ A $, we write 
\begin{align*}
	\sum_{A}f:=\sum_{a\in A} f(a)
\end{align*}
as it is commonly used for integrals. Typically, the set $ A $ will be a subset of either $ X $ or $ X\times X $.  A  function $m:X\to [0,\infty)  $ extends to a measure on $ X $ by letting $$  m(A)=\sum_{A}m,\qquad A\subseteq X  .$$ 
Furthermore, for $ q\in[1,\infty) $, we denote
\begin{align*}
	\ell^{q}(X,m)=\{ f:X\to \R\mid  \|f\|_{q}:=\big(\sum_{X}m|f|^q\big)^{1/q} <\infty \},
\end{align*}
and  $ \ell^{\infty}(X)= \{f:X\to \R\mid  \|f\|_{\infty}:=\sup_{X}|f|<\infty\}$.

In what follows we assume that $ m $ is a \emph{measure of full support}, i.e. the {density} function $ m $ is strictly positive.  However, we will still consider $ \ell^q $ spaces for measures which may not have full support in what follows, e.g. by multiplying $ m $ with a density which may vanish somewhere.

With slight abuse of notation, for a function $ w:X\to\R $, we denote the functional
\begin{align*}
	w(\ph):=\sum_{X}mw |\ph|^{p}
\end{align*}
 on $ C_{c}(X) $ also by $ w $.

\subsection{Graphs, energy functionals and $p$-Schrödinger operators}

Let $b$ be a connected graph over $(X,m)$, i.e. $ b:X\times X\to [0,\infty) $ is symmetric, has zero diagonal,  satisfies
\begin{align*}
	\sum_{y\in X}b(x,y)<\infty\qquad x\in X,
\end{align*} 
and for every $ x,y\in X $ there are $ x =x_{0}\sim\ldots \sim x_{n}=y $, where we write $ u\sim v $ whenever $ b(u,v)>0 $ and call $ u $ and $ v $ \emph{adjacent}. A graph $ b $ is called \emph{locally finite} if $ \#\{y\in X\mid y\sim x\} <\infty$ for all $ x \in X$. Although we do not need this assumption for the main results of this paper, it is worthwhile from time to time to look at the special case of locally finite graphs.

Let $ c: X\to \R $ be a given {\em potential}. We introduce the energy functional $ \Q=\Q_{p,b,c} $ on $ C_{c}(X) $ {given by}
$$\Q(\varphi)=\frac{1}{2}\displaystyle \sum_{x,y \in X}  b(x,y)|\nabla_{xy}\varphi|^p + \sum_{x \in X} c(x)|\varphi(x)|^p,  \qquad \varphi \in  C_c(X), $$
where,  for $ x,y\in X $ and $ f\in C(X) $,
$$ \nabla_{xy}f:=f(x)-f(y). $$  We assume throughout the paper that 
\begin{align*}
	\Q\ge 0, 
\end{align*} 
which means $ \Q (\ph)\ge 0$ for all $ \ph\in C_{c}(X) $. A crucial feature of $ \Q $ is that it is compatible with the absolute value, i.e for $ \ph\in C_{c}(X) $ 
\begin{align*}
	\Q(|\ph|)\leq \Q(\ph).
\end{align*}
From the energy functional $ \Q $, the
{\em weighted $p$-Laplacian} $ \L =\Delta_{p,b,c,m}  $ arises. It is acting on the following space of functions
\begin{align*}
	\mathcal{F}(X):=\{ f\in C(X)\mid \sum_{y \in X} b(x,y) |\nabla_{xy} {f}|^{p-1}<\infty\mbox{ for all }x\in X  \}
\end{align*} 
as
$$ \L f(x):=\frac{1}{m(x)}\sum_{y \in X} b(x,y) |\nabla_{xy} f|^{p-2}(\nabla_{xy} f). $$ 
Clearly, in the case of locally finite graphs one has $ \mathcal{F}(X)=C(X) $. For a potential $ g:X\to \mathbb{R} $, we denote for $ \varphi\in C(X) $ and $ x\in X $
\begin{align*}
	g[\varphi](x):=g(x)|\varphi(x)|^{p-2}\varphi(x).
\end{align*}

\begin{rem}
	The  counterpart of $\L$ in 
	the continuum case is referred to as the {\em pseudo $p$-Laplacian} which acts on sufficiently smooth real valued functions $ \varphi $ defined on an open subset of the Euclidean space by  $\varphi \mapsto -\sum_{i=1}^n\partial_i(|\partial_i\varphi|^{p-2}\partial_i\varphi) $
	and which is sometimes denoted in the literature by $ \tilde\Gd_p $, see \cite{BK} and references therein. 
\end{rem}

We consider the  quasilinear homogeneous equation for the operator $ \QQ=\L+c/m$
\begin{align*}\label{e:Q'} \tag{$ \QQ $}
	\QQ[\varphi]:=\L \varphi + \frac{c}{m}|\varphi|^{p-2}\varphi=0
\end{align*}
and call $ u\in \mathcal{F} (X)$ a \emph{(super-)solution} of \eqref{e:Q'} on $ Y\subseteq X $ if $ \QQ[u]=0 $ (respectively, $ \QQ[u]\ge 0 $) on $ Y $. We recall {the} \emph{Agmon-Allegretto-Piepenbrink}-type theorem \cite[Theorem~2.3]{Florian_cri}, which states that  $\Q\geq 0$ on $C_c(X)$ if and only if  \eqref{e:Q'} admits a positive supersolution on $X$.  In the case of {locally} finite graphs,  $\Q\geq 0$ on $C_c(X)$  even implies the existence of a positive solution on $X$. By the Harnack inequality \cite[Lemma~4.4]{Florian_nonlocal} it can be seen that due to the connectedness, every positive supersolution is indeed strictly positive.
\medskip

\begin{rem}[Restriction to subsets]  Let us discuss the case of solutions on subsets  $ Y\subseteq X $. If $ u \in \mathcal{F}(X) $ satisfies $ u=0 $ on $ X\setminus Y$ and $ \QQ[u]=f $ on $ Y $ for some $ f:X\to\R $, then one gets by a direct computation that the restriction $ u_{Y}:=u\vert_{Y} $ satisfies $ \QQ_{Y}[u_Y] =f_{Y}$. Here $ \QQ_{Y} $ is the operator which arises from the restriction $ b_{Y} $ of the graph $ b $ to $ Y\times Y $ and the potential
	\begin{align*}
		c_{Y}(y):=c(y)
		+\sum_{x\in X\setminus Y}b(x,y),\qquad y\in Y.
	\end{align*}
Therefore, this problem of solutions on subsets is included within our setting.
However, in contrast to the linear case $ p=2 $ such a reduction is not possible for $ u $ being equal to some nontrivial function $ g $ outside of $ Y $. For $ p=2 $, the equation $ \QQ[u] =f$ on $ Y $ then reduces to   $ \QQ_{Y}[u_Y] =f_{Y}+g'_{Y}$ where $ g'_{Y}(x):=\sum_{y\in X\setminus Y}b(x,y)g(y) $. This is no longer the case for $ p\neq 2 $.
\end{rem}

Let us emphasize that we do not assume that our graphs are locally finite. This allows us to include our main results of this article for {\em fractional Laplacians} and the corresponding Schr\"odinger operators on graphs .
\begin{example}[Fractional $ p $-Laplacian] \label{Exam:frac} In the case $ p=2 $, Hardy inequalities for the discrete fractional Laplacian 
\begin{align*}
	\Delta_{2}^{\sigma}f(x):=\frac{1}{|\Gamma(-\sigma)|}\int_{0}^{\infty}(I-e^{-t\Delta_{2}})f(x)\frac{dt}{t^{1+\sigma}}
\end{align*}	
	 for $ \sigma \in (0,1) $ are studied in \cite{CiaurriRoncal,KellerNietschmann}. For the Laplacian $ \Delta_{2}=\Delta_{2,b,0,m}  $ on a connected graph $ b $ over $ (X,m) $, the fractional Laplacian is again a graph Laplacian $  \Delta^{\sigma}_{2}  =\Delta_{2,b_{\sigma},0,m} $ for a  weighted graph $ b_{\sigma} $.  Since the semigroup $ e^{-t\Delta_{2}} $ maps positive functions to strictly positive functions on connected graphs, $ b_{\sigma} $ will be complete, i.e. every two vertices are adjacent. Hence, whenever the original graph is infinite, then the graph $ b_{\sigma} $ will be non-locally finite. One can proceed to study the $ p $-fractional Laplacian which is given for suitable functions $ f $  as
	\begin{align*}
	\Delta_{p}^{\sigma}f(x)&:=C\int_{0}^{\infty}e^{-t\Delta_{2}}\left( |\nabla_{x(\cdot)} f|^{p-2}(\nabla_{x(\cdot)} f)\right)\frac{dt}{t^{1+\frac{\sigma p}{2}}}=C \sum_{y\in X}b_{\sigma,p}(x,y) |\nabla_{xy} f|^{p-2}(\nabla_{xy} f)	\end{align*}
with \begin{align*}
b_{\sigma,p}(x,y):= \int_{0}^{\infty}e^{-t\Delta_{2}}1_0(x-y)\frac{dt}{t^{1+\frac{\sigma p}{2}}}
\end{align*}
and some suitable choice of the constant $C  $, cf. \cite{DelTeso}, which gives again rise to a $ p $-Laplacian for a non-locally finite graph. Since our main results are valid for non-locally finite graphs, these operators are included in the set-up of our paper.
\end{example}

We close this subsection with a few words about notation. With {a} slight abuse of notation we identify constants with the corresponding constant function at times which is mainly applied to the constants $ 0 $ and $ 1 $. For example we write  $ f\ge 0 $ to indicate that a function is pointwise larger or equal than $ 0 $. In the case when $ f $ is additionally not trivial, then we say $ f $ is \emph{positive}. Moreover, if $ f>0 $, then we say $ f $ is \emph{strictly positive}.  We denote the characteristic function of a set $ A $ by $ 1_{A} $ and also write $ 1=1_{X} $. For real valued functions $ f $ and $g$, we write $ f\wedge g=\min\{f,g\} $ and $ f\vee g=\max\{f,g\} $,  and we denote the positive and negative part of $f$ by $ f_{\pm}=({\pm}f)\vee0 $.  
Finally, in this text, $ C $ denotes a constant which depends only on $ p $ and may change from line to line.
%
%%%%%%%%%%%
\subsection{The simplified energy functional}

We introduce the simplified energy which serves as a remedy for the absence of the ground state transform in the nonlinear case.

\begin{definition}[Simplified energy] \label{Def:SEF}
	 Let ${\sol{u}} \in  C(X)$ be a positive function.  For $ \varphi\in C_{c}(X) $, let 
			$$ \S{\varphi}_{xy} := b(x,y) {\sol{u}}(x) {\sol{u}}(y) |\nabla_{xy} \varphi|^2 \left[ |\nabla_{xy} {\sol{u}}|\frac{|\varphi(x)|+|\varphi(y)|}{2} + ({\sol{u}}(x) {\sol{u}}(y))^{1/2} |\nabla_{xy} \varphi|\right]^{p-2} $$
			 where we set $0 \cdot \infty =0$ if $1 < p < 2$,
		and define the {\em simplified energy functional}  
		of the functional $\Q$ with respect to  $\sol{u}$  as
		$$\E(\varphi):= \sum_{x,y\in X} \S{\varphi}_{xy} . $$ 
\end{definition}
For functions $ f ,g$ and vertices $ x,y $, we introduce the notation
\begin{align*}
	(f\otimes g)_{xy}=f(x)g(y)\qquad \mbox{and}\qquad \langle f\rangle_{xy}=\frac{f(x)+f(y)}{2}.
\end{align*}
With these conventions the terms $ \S{\ph}:X\times X\to [0,\infty) $ can be expressed more concisely as 
\begin{align*}
	\S{\ph}= b ({\sol{u}}\otimes {\sol{u}})^{p/2} |\nabla \varphi|^2 \left[ \frac{|\nabla {\sol{u}}|}{({\sol{u}}\otimes {\sol{u}})^{1/2}} \langle{|\varphi|}\rangle+  |\nabla \varphi|\right]^{p-2} \,.
\end{align*}

%-----------------------------------------------------------------------------
%
%Alternatively one could also write
%\begin{align*}
%	b_{\sol{u}}(x,y)=b(x,y)(\sol{u}(x)\sol{u}(y))^{p/2},\qquad \gamma(\sol{u})_{xy}=\frac{|\nabla_{xy}\sol{u}|}{(\sol{u}(x)\sol{u}(y))^{1/2}}
%\end{align*} 
%to write
%\begin{align*}
%	\S{\ph}= b_{\sol{u}} |\nabla \varphi|^2 \left(\gamma(\sol{u}) \langle{|\varphi|}\rangle + |\nabla \varphi|\right)^{p-2} 
%\end{align*}
%in which case one does not need $ f \otimes g$
%
%
%-----------------------------------------------------------------------

In contrast to the ground state transform in the linear case,  the simplified energy is not a representation of  the energy functional $ \Q $. However, we still have a two sided estimate which goes back to {Pinchover/Tertikas/Tintarev \cite{ptt}} in the continuum setting and to Fischer \cite{Florian_nonlocal} in the discrete setting.
%%%
\begin{proposition}[Simplified energy, {\cite[Theorem 3.1]{Florian_nonlocal}}] \label{Prop:simp_energy} There are $C_1,C_2>0$ such that  for all strictly positive functions  $ \sol{u} \in \mathcal{F}(X)$  and for all  $ \varphi  \in  C_c(X) $
	$$C_1 \Q({\sol{ u}}\varphi) \leq  \E(\varphi)+ \sum_{X}mu\QQ[u]||\varphi|^{p} \leq C_2 \Q({\sol{u}}\varphi) .$$
\end{proposition}

Generally speaking, in the nonlinear case,  the nonnegative energy $Q$ is not convex, and might contain negative terms. On the other hand, the simplified energy contains only nonnegative terms yet  in general it is not a convex functional. Still, in certain situation (e.g. proof of Theorem \ref{Thm:Char}), it turns out that the simplified energy plays an efficient role due to its compatibility with certain normal contractions as it is shown in the following lemma. 

A \emph{normal contraction} is a function $ \mathcal{C}:\R\to \R $ which satisfies $ \mathcal{C}(0) =0 $ and for $ s,t\in \R $
\begin{align*}
	|\mathcal{C}(s)-\mathcal{C}(t)|\leq |s-t|.
\end{align*}
{If $\mathcal{C}$ is a normal contraction and $f:X\to \R$, we say that $\mathcal{C}\circ f$ is a {\em normal contraction of $f$}.}

\begin{lemma}\label{lem:contr} Let $ \sol{u} $ be a strictly positive function. Then, for all $ \varphi\in C_{c}(X) $, we have
	\begin{align*}
		\S{|\ph|}\leq \S{\ph}
	\end{align*}
	and for all $ \alpha,\beta\ge0 $ 
	\begin{align*}
		\S{(-\alpha) \vee \ph \wedge \beta }\leq 	\S{\ph }.
	\end{align*}
Moreover, if $ p\ge 2 $, then 	we have $$  \S{\mathcal{C}\circ\ph}\leq \S{\ph}  $$
for all normal contraction $  \mathcal{C}$ and $ \ph\in C_{c}(X) $.
\end{lemma}
\begin{remark}
	Clearly, the lemma implies that the above properties of $S_u$ hold for $E_u$. 
\end{remark}
\begin{proof}[Proof of Lemma~\ref{lem:contr}]
The last statement is immediate as $ |\nabla  (\mathcal{C}\circ\ph)|\leq |\nabla \ph|$ and $ \langle{ |\mathcal{C}\circ\ph|}\rangle \leq \langle{ |\ph|}\rangle $. As $ |\cdot| $
and $ (-\alpha \vee (\cdot) \wedge \beta )$ are normal contractions, from now on, we let $ p< 2 $.

For the first statement, let  $ U\ge 0 $ and 	$ \Psi :[0,\infty)\to [0,\infty) $, $ \Psi(s)=s^2(U + s)^{p-2} $.
Then,
\begin{align*}
	\Psi'(s) = 2s(U+s)^{p-2}+ (p-2)s^2(U+s)^{p-3} =2s(U+s)^{p-2}\left(1- \left(1-\frac{p}{2} \right)\frac{s}{U+s}\right) \ge 0
\end{align*}
since $ s/(U+s)\leq 1  $ and $ (1-p/2)\le 1/2 $.
With $ U=\langle {|\ph|}\rangle|\nabla \sol{u}|/(\sol{u}\otimes \sol{u})^{{\frac{1}{2}}}$, we get since
$ |\nabla_{xy}  |\ph|| \le |\nabla_{xy}  \ph|$
 \begin{align*}
	\frac{\S{|\ph|}_{xy} }{b(x,y)(\sol{u}(x)\sol{u}(y))^\frac{p}{2} }= \Psi(|\nabla_{xy}  |\ph||)\leq  \Psi(|\nabla_{xy}  \ph|) =
		\frac{\S{\ph}_{xy} }{b(x,y)(\sol{u}(x)\sol{u}(y))^\frac{p}{2} }
\end{align*}
which proves the first statement.

To prove the second inequality, 
we show for $ \ph\in C_{c}(X) $
	\begin{align*}
	\S{ \ph \wedge 1 }\leq 	\S{\ph }.
\end{align*}
From this we conclude $ 	\S{ \ph \wedge \beta }\leq 	\S{\ph } $  for $ \beta> 0  $ by homogeneity of $S_u$. Similarly, for $\ga>0$ we obtain 	
$$  \S{(-\alpha)\vee\ph} = \S{-[\alpha\wedge -\ph]}  =\S{\alpha\wedge (-\ph)} \leq 	\S{-\ph } =	\S{\ph }.$$
For $ \alpha,\beta=0 $ the statement then follows by taking the limit. It remains to  show that  $ 	\S{ \ph \wedge 1 }_{xy}\leq 	\S{\ph }_{xy} $ for $ \ph\in C_{c}(X) $ and $ x,y\in X $. 

We can assume without loss of generality $ \ph(x)<1< \ph(y)  $ due to the symmetry and since the inequality  is trivial otherwise. Let  $ U>0 $	 and $ \Phi :\R^{2}\to [0,\infty) $ be given as
\begin{align*}
	\Phi (s,t) =  U\frac{|s|+|t|}{|s-t|}+1. 
\end{align*}
For {$ t<1<s$,} we have
\begin{align*}
	\partial_{s}\Phi (s,t) = U	\partial_{s} \left(\frac{s+|t|}{s-t}\right) =U\frac{(s-t)-s-|t|}{(s-t)^2} =
	U\frac{-2t\wedge 0}{(s-t)^2} \leq 0.
	\end{align*}
With $ U= |\nabla_{xy} \sol{u}|/(\sol{u}(x) \sol{u}(y))^{{\frac{1}{2}}}$, we get since
$ |\nabla_{xy}  (\ph\wedge 1)| \le |\nabla_{xy}  \ph|$ and $ p< 2$ 
\begin{align*}
	\frac{\S{(\ph\wedge 1)}_{xy} }{b(x,y)(\sol{u}(x)\sol{u}(y))^\frac{p}{2} }= 
	\frac{|\nabla_{xy}  (\ph\wedge 1)|^{p} }{\Phi(1,\ph(x))^{2-p}}
\leq  
\frac{|\nabla_{xy}  \ph|^{p} }{\Phi(\ph(y),\ph(x))^{2-p}}=
	\frac{\S{\ph}_{xy} }{b(x,y)(\sol{u}(x)\sol{u}(y))^\frac{p}{2} }.
\end{align*}
This finishes the proof.
\end{proof}

An immediate consequence of the lemma above is the following statement.
\begin{lemma}\label{l:cuttoff} There is $ C>0 $ such that for all strictly positive supersolutions  $ \sol{u} $  of \eqref{e:Q'} and for all $ \varphi\in C_{c}(X) $, we have
	\begin{align*}
		\Q(0\vee \ph\wedge \sol{u})\leq C\Q(\ph).
	\end{align*}
\end{lemma}
\begin{proof}
	By the simplified energy, Proposition~\ref{Prop:simp_energy}, and the lemma above we get with $ f=u\QQ[u]\ge 0$ and $\psi=\ph/u$
	\begin{align*}
	 C_{1}	\Q(0\vee \ph\wedge \sol{u})  \leq \E (0\vee \psi\wedge 1)+\sum_{X}m f |0\vee \psi\wedge 1|^{p} \le\E(\psi)+\sum_{X}m f | \psi|^{p}\leq    C_{2}\Q(\ph).
	\end{align*}
This completes our proof.
\end{proof}

\subsection{Generalized capacity} \label{subsec-cap}
Inspired by the previous work \cite{Das_Pinchover1}, we extend the classical definition of $p$-capacity on compact sets in $X$ to the case of the nonnegative functional $\Q$. 
\begin{definition}[$\Q$-capacity] \label{Def:Cap} {\em Let ${\sol{u}} \in  C(X)$ be a {strictly} positive function.  
		For a {compact set $K$ in $X$,} the {\em  $\Q$-capacity of $K$ with respect to ${\sol{u}}$} is defined by
		%	$${\rm{Cap}}_{\sol{u}}(K):=\inf \{\Q(\varphi) \mid\varphi \in \mathcal{N}_{K,\sol{u}}(X)\} \,,$$
		%	where $\mathcal{N}_{K,\sol{u}}(X):=\{\varphi \in   C_c(X) \mid \varphi \geq {\sol{u}} \ \text{on} \ K\}$. 
		\begin{align*}
			{\rm{Cap}}_{\sol{u}}(K):=\inf \{\Q(\varphi) \mid\varphi \in C_{c}(X),\; \ph\geq 1_{K}{\sol{u}} \} \,.
		\end{align*}
	}
\end{definition}
Let us briefly discuss some of the  properties of ${\rm{Cap}}_{\sol{u}}$ (cf. \cite{Das_Pinchover1} and references therein).

%%%%%%%%%%% 
\begin{lemma} \label{Cap_def_eqiv} Let ${\sol{u}} \in  C(X)$ be a {strictly} positive function and a compact set  $K $ in  $ X$. Then
	$$ {\rm{Cap}}_{\sol{u}}(K)=\inf \{\Q(\psi  {\sol{u}}) \mid \psi \in  C_c(X), \; \psi \geq 1_{K} \,\}. $$
\end{lemma}
\begin{proof}
	%	(a) 
	This is immediate as $ \psi \sol{u} \ge \sol{u} $  { on $ K $} if and only if $\psi\ge 1 $ on $ K $.	
	%	(b) One can easily verify that $\varphi \in \mathcal{N}_{K,\sol{u}}(X)$ implies that $|\varphi| \in \mathcal{N}_{K,\sol{u}}(X)$ and $\Q(|\varphi|) \leq \Q(\varphi)$. Hence, in the definition of the $\Q$-capacity, it is enough to  consider only $\varphi \in \mathcal{N}_{K,\sol{u}}(X)$ with $ \varphi\ge 0 $. Furthermore, for any $\varphi \in \mathcal{N}_{K,\sol{u}}(X)$,   we have ${\varphi}\wedge \sol{u}=\in \mathcal{N}_{K,\sol{u}}(\Om)$, and if $c \geq 0$ then $\Q({\varphi}\wedge\sol{u}) \leq \Q(\varphi)$ (using the inequality $|\al \wedge \gamma -\beta \wedge \gamma | \leq |\al -\beta|$ for any $\al,\beta, \gamma \in \R$). 	Thus, the statement follows.
\end{proof}

\begin{remark}
	In view of the lemma above, when $c \geq 0$ by choosing ${\sol{u}} = 1$, we see that our definition of $\Q$-capacity coincides with the definition in \cite{Florian_cri}. {Below, we discuss the case of general potential $ c $.}
\end{remark}
%%%%%%%%%%%%%%%%%%%%%%%%%%%%%	
Next, we discuss two alternative versions of the generalized capacity and show that they are equivalent to the one above.
First, {fix a strictly positive $u\in C(X)$.} For a compact set $ K$ in  $ X $, let
\begin{align*}
	\widetilde{{\rm{Cap}}}_{\sol{u}}(K):=\inf \{\Q(\ph) \mid   \ph \in C_{c}(X),\;\sol{u} \geq \ph\geq  1_K \sol{u}\} 
\end{align*}
and secondly
\begin{align*}
	{\rm{Cap}}_{{\sol{u}},{\rm{Sim}}}(K):=\inf\{E_{{\sol{u}}}(\varphi)+\sum_{X}mu\QQ[u]|\varphi|^{p} \mid \varphi \in C_c(X), \ \varphi \geq 1_K \} \,.
\end{align*}
For positive functions $ F $ and $ G $, we write $ F\asymp G $ if there is a constant $ C>0 $ such that $ C^{-1}F \leq G\leq C F$.

\begin{lemma}\label{lem:cap} Let $ \sol{u} $ be a strictly positive supersolution of of \eqref{e:Q'}. Then,
	\begin{align*}
{\rm{Cap}}_{{\sol{u}}}  \asymp	\widetilde{{\rm{Cap}}}_{\sol{u}}\asymp	{\rm{Cap}}_{{\sol{u}},{\rm{Sim}}}	.
	\end{align*}
Furthermore, $\cp(K)=0$ if and only if ${\rm{Cap}}_{{{1}}}(K)=0$  for compact
$K $ in  $ X$.
\end{lemma}
\begin{proof}
	The equivalence $ {\rm{Cap}}_{{\sol{u}}} \asymp		{\rm{Cap}}_{{\sol{u}},{\rm{Sim}}}   $ follows from the simplified energy, Proposition~\ref{Prop:simp_energy}.  
	The inequality $ {\rm{Cap}}_{{\sol{u}}}  \leq 	\widetilde{{\rm{Cap}}}_{\sol{u}}$ is trivial. The reverse inequality $ {\rm{Cap}}_{{\sol{u}}}  \geq C	\widetilde{{\rm{Cap}}}_{\sol{u}}$ for some $ C>0 $ follows by  Lemma~\ref{l:cuttoff} above.
	
	For the second statement, assume first that ${\rm{Cap}}_{{{1}}}(K)=0$. Consider ${\sol{u}}_K={{\sol{u}}}/{\|{\sol{u}} 1_K\|_{{\infty}}}$. Then, it is clear that ${\rm{Cap}}_{{\sol{u}}_K}(K)\leq {\rm{Cap}}_{{{1}}}(K).$ This implies  $$ \cp(K)=\frac{1}{\|{\sol{u}}1_K\|_{{\infty}}^p} {\rm{Cap}}_{{\sol{u}}_K}(K)\leq \frac{1}{\|{\sol{u}}1_K\|_{{\infty}}^p}{\rm{Cap}}_{{{1}}}(K) =0. $$  Conversely, suppose that $\cp(K)=0$, and consider $\widetilde{{\sol{u}}}_K= {{\sol{u}}}/{\inf_{K}{\sol{u}}}$. Then, following a similar argument, we conclude that ${\rm{Cap}}_{{{1}}}(K)=0$.	
\end{proof}

%Above we have already seen that the simplified energy is compatible with all normal contractions for $ p\ge2  $ and with certain normal contractions for $ 1<p<2 $. Below we show we can say a bit more even in the case $ 1<p<2 $.

%%%%%%%%%%%%%%%%%%%%%%%%%%%%%%%%
%%%%%%%%%%%%%%%%%%%%%%%%%%%%%%
\section{The space of Hardy-weights} \label{sec_Mazya}
 In this section we study a Maz'ya type criterion to characterize the Hardy weights in terms of the generalized $ p $-capacity. This extends the works of \cite{Das_Pinchover1,Dyda,Mazya2} to discrete graphs.
 
{A function  $g$ which satisfies for some $ C>0 $ the following Hardy-type inequality  
 \begin{align*}
 C \sum_{ X}  m|g| |\varphi|^p \leq  \Q(\varphi) 
 \end{align*}
 for all $ \varphi \in  C_c(X)    $ is called a {\em Hardy-weight} for $Q$, in which case we write 
 $$ C |g|\leq \Q\;\mbox{ on  }\;C_{c}(X) . $$ 
The first aim of the present section is to characterize the space of all Hardy-weights. }
 We denote the space of all Hardy-weights  by 
$$\H =\mathcal{H}_p(X,m,b,c):=\{g \in C(X) \mid C |g| \leq \Q\mbox{ for some }C>0\}.$$
The nonnegative functional $\Q$ is said to be {\em subcritical (resp., critical) in $X$} if the space  $\H\neq \{0\}$ (resp., $\H = \{0\}$). 
\begin{remark}\label{rem_crit} 
(a) Using a partition of unity argument, it follows that if $\Q$ is subcritical in $X$, then there exists $0<w\in \H\cap L^\infty(X,m)$ (see \cite[Corollary 5.6]{Florian_cri}).

(b) $\Q$ is critical in $X$ if and only if $\QQ[\varphi]=0$ admits a unique positive supersolution (up to a multiplicative positive constant) \cite{Florian_cri}. In fact,  such a supersolution is a strictly positive solution and is called an {\em Agmon ground state}.
\end{remark}

The best constant $C\ge 0$ that satisfies the  inequality $ C|g|\leq \Q $ is called the {\em Hardy constant} of a function $ g $ and it is denoted by $ C(g) $. By definition,  $ g \in  \mathcal{H} $ if and only if $ C(g)>0 $.
We denote the reciprocal of the {Hardy constant} $C(g)$  by {$\norm{g}_{\H}$, i.e.} 
  \begin{align*}
  	\norm{g}_{\H}=\sup_{\ph\in C_{c}(X),\Q(\ph)\neq 0} \frac{\sum_{X}m|\ph|^{p}|g|}{\Q(\ph)}=\frac{1}{C(g)}.
  \end{align*}
It is immediate to see that $ \norm{\cdot}_{\H} $ is a norm on $ \H $. Furthermore, let ${\sol{u}} \in  C(X)$ be a positive function. Recalling the definition  of ${\rm{Cap}}_{\sol{u}}(K)$ (see Definition~\ref{Def:Cap}), i.e.  the  $\Q$-capacity of a compact set 
    $K \subset X$ with respect to ${\sol{u}}$, we define for $ g\in C(X) $
\begin{eqnarray*}
	\norm{g}_{\H,\sol{u}}:=\sup_{K \Subset X, \ \cp(K)\neq 0} \frac{\sum_{ K} m{\sol{u}}^p|g| }{\cp(K)} .
 \end{eqnarray*}
Clearly, $0\leq\norm{g}_{\H,\sol{u}}\leq \infty$. Again it is immediate to see that $ \norm{\cdot}_{\H,\sol{u}} $ is a norm, however it is not a priori clear that it takes finite values on $ \H $ and infinite values outside of $ \H $. For {positive supersolutions} $u$ of \eqref{e:Q'},  we  prove {in the following theorem that the norms $ \norm{\cdot}_{\H} $ and $ \norm{\cdot}_{\H, \sol{u}} $ are equivalent.} This shows, in particular, that up to equivalence of norms $ \norm{\cdot}_{\H,\sol{u}} $ is independent of $ \sol{u} $. The proof is inspired by the proofs of \cite[Theorem~1.2]{Das_Pinchover1} and \cite[Proposition~5]{Dyda} but utilizes new ideas for the compatibility of the simplified energy with certain normal contractions.

\begin{theorem}[Maz'ya-type characterization] \label{Thm:Char} 	There is a  constant $ C_{p} $ depending only on $ p $  such that for all functions $ g $ and all strictly positive  supersolutions $ \sol{u} $ of \eqref{e:Q'} 
	%with $ 0\leq f\in C_{c}(X) $
	\begin{align*}
		\norm{g}_{\H,\sol{u}} \leq \norm{g}_{\H}\leq C_p \norm{g}_{\H,\sol{u}}.
	\end{align*}
%A function $ g $ satisfies $\norm{g}_{\H,\sol{u}}<\infty$ if and only if $g$ is a Hardy-weight of $\Q$.  Moreover, let $\B:={B}_g(X,m,b,c)$ be the best constant in $ g\lesssim \Q $, then $$\norm{g}_{\H,\sol{u}} \leq \B \leq C_p \norm{g}_{\H,\sol{u}},$$ where $C_p$ depends only on $p$. Furthermore, $\|g\|_{B} :=\B$ is an equivalent norm on $\H$. In particular, up to the equivalence relation of norms, the norm $\norm{\cdot}_{\H_p^{\sol{u}}(X,m,c)}$ is independent of the positive solution ${\sol{u}}$.
\end{theorem}
\begin{proof}
{$ \norm{g}_{\H,\sol{u}} \leq \norm{g}_{\H} $:} Let $g \in\H$  and  $K \subset X$ be a compact set. Then for all $\ph \in  C_c(X)$ with $\ph \geq 1_K u$,
\begin{align*}
 \sum_{ K} m {\sol{u}}^p |g|  \leq \sum_{ X} m  |\ph|^p| g| \leq  \norm{g}_{\H} \Q(\ph) \,.
\end{align*}
By taking infimum over all such $ \ph $, we obtain that for all compact sets  $K$ in $X$
\begin{align*} 
 \sum_{ K} m {\sol{u}}^p |g|  \leq  \norm{g}_{\H}     {\rm{Cap}}_{\sol{u}}(K) .
\end{align*}
Hence, $\|g\|_{\H,\sol{u}} \leq \norm{g}_{\H} $ and we are left to show the second inequality.\medskip

{$ \norm{g}_{\H}\leq C_p \norm{g}_{\H,\sol{u}} $:} Let $ g\in C(X) $ be  such that $\|g\|_{\H,\sol{u}}<\infty$, i.e. for all compact sets  $K$
\begin{align*}%\label{Mazya_Cond}
	\sum_{ K}  m {\sol{u}}^p|g| \leq \|g\|_{\H,\sol{u}} {\rm{Cap}}_{\sol{u}}(K).
\end{align*}
 Let $ \psi\in C_{c}(X) $, $ \psi\ge 0 $, and denote the following level-set annuli for $ k\in \mathbb{Z} $ by
\begin{align*}
	A_k:=\{2^{k}< \psi\le2^{k+1}\}\quad\mbox{and}\quad
	 B_{k}:= A_{k-1}\cup A_{k}\cup A_{k+1}.
\end{align*}
Then,
\begin{align*}
	\sum_{X}m|\psi{\sol{u}}|^p|g|& = \sum_{k\in\mathbb{Z}}\sum_{A_{k}}m|\psi{\sol{u}}|^p |g|  \leq \sum_{k\in\mathbb{Z}}\sum_{A_{k}}2^{(k+1)p}m|{\sol{u}}|^p|g| \nonumber\\
	&\leq \|g\|_{\H,\sol{u}}   \sum_{k \in \mathbb{Z}} 2^{(k+1)p}  {\rm{Cap}}_{\sol{u}}(A_{k}).
%	\leq 2^{2p} \|g\|_{\H,\sol{u}}   \sum_{k \in \mathbb{Z}} 2^{kp}  {\rm{Cap}}_{\sol{u}}(\Omega_{k+1})	\,.
\end{align*}
We outsource the following claim
\begin{align*} 
	{\rm{Cap}}_{\sol{u}}(A_{k}) &\leq 
	2^{-kp}C  \left(\sum_{B_{k}\times X}  \S{\psi}+\sum_{B_{k}}m u\QQ[u] |\psi|^{p} \right).
\end{align*}
 to Lemma~\ref{Lem:ujjalp>2}  below and use it together with the estimate above and the simplified energy, Proposition~\ref{Prop:simp_energy}, to conclude
\begin{align*}
	\sum_{X}m|g||\psi{\sol{u}}|^p&\leq C \|g\|_{\H,\sol{u}}   \sum_{k \in \mathbb{Z}} 2^{kp}  {\rm{Cap}}_{\sol{u}}(A_{k})\leq C \|g\|_{\H,\sol{u}}   \sum_{k \in \mathbb{Z}}   \left(\sum_{B_{k}\times X}  \S{\psi}+\sum_{B_{k}}m u\QQ[u] |\psi|^{p} \right)\\
	&\le C\|g\|_{\H,\sol{u}}  \left(\sum_{X\times X}  \S{\psi}+\sum_{X}m u\QQ[u] |\psi|^{p} \right)  \leq C\|g\|_{\H,\sol{u}}   \Q(\psi u).
\end{align*}
Hence, for $ \ph\in C_{c}(X) $, we let $ \psi =|\ph|/u $ and observe
\begin{align*}
	\sum_{X}m|g||\ph|^p=	\sum_{X}m|g||\psi{\sol{u}}|^p\leq C\|g\|_{\H,\sol{u}}   \Q(\psi u)= C\|g\|_{\H,\sol{u}}   \Q(|\ph|)\leq C\|g\|_{\H,\sol{u}}   \Q(\ph),
\end{align*}
where we used Lemma~\ref{lem:contr} in the last estimate. This finishes the proof. 
\end{proof}
The following lemma settles the claim used in the proof of the theorem above to estimate $ {\rm{Cap}}_{\sol{u}}(A_{k})  $  by the simplified energy.

\begin{lemma}\label{Lem:ujjalp>2}   For $1<p<\infty$, there exists $C>0$ such that for all strictly positive  supersolutions $ \sol{u} $ of \eqref{e:Q'} and  $0 \leq \psi\in C_{c}(X) $, we have  
	\begin{align*} 
		{\rm{Cap}}_{\sol{u}}(A_{k}) &\leq 
		2^{-kp}C  \left(\sum_{B_{k}\times X}  \S{\psi}+\sum_{B_{k}}m u\QQ[u] |\psi|^{p} \right) \ \ \forall k\in\mathbb{Z} \,,
	\end{align*}
	where $ A_k=\{2^{k}< \psi\le2^{k+1}\}$, $ B_{k}= (A_{k-1}\cup A_{k}\cup A_{k+1}) $.
\end{lemma}
\begin{proof}
We treat the cases $ p\ge 2 $ and $ p<2 $ separately.

\emph{Case $ 2\le p <\infty $}: For  positive $ \psi\in C_{c}(X) $, consider the function $\psi_k \in  C_c(X)$ given by
\begin{align*}
	\psi_k&:=0\vee\left[ (2^{-k+1}\psi-1)\wedge (4-(2^{-k+1}\psi-1))  \right]\wedge 1\\
&	=
\begin{cases}
	2^{-k+1}\psi -1 & \mbox{if $ 2^{k}\ge \psi>2^{k-1} $, i.e. on $  A_{k-1} $}, \\
	1& \mbox{if $ 2^{k+1}\ge \psi>2^{k} $, i.e. on $  A_{k} $} ,
	\\
4-(2^{-k+1}\psi-1) & \mbox{if $ 2^{k+1}+2^{k-1}\ge \psi>2^{k+1} $, i.e. on a subset of $  A_{k+1} $},\\
0& \mbox{ else}.
\end{cases}
	\end{align*}
Observe the fact that $ \psi_{k} $ is a normal contraction of $ 2^{-k+1}\psi $. Thus, we have by Lemma~\ref{lem:contr} 
\begin{align*}
\S{\psi_{k}}\leq \S{2^{-k+1}\psi}=2^{-(k-1)p}\S{\psi}.
\end{align*}
Furthermore,   $ |\nabla_{xy} \psi_{k}| $, and therefore also $ \S{\psi_{k}}_{xy} $, vanish outside of $(B_{k}\times X)\cup (X\times B_{k}
)$. Moreover,   $ \psi_{k} $ is supported on $ B_{k} $ and $ \psi_{k} \le 2^{-k+1}\psi  $.  Hence, by Lemma~\ref{Cap_def_eqiv} and the simplified energy, Proposition~\ref{Prop:simp_energy}, we estimate
\begin{align*} 
		 {\rm{Cap}}_{\sol{u}}(A_{k}) &\leq  \Q(\psi_k {\sol{u}}) \leq  C \left(\sum_{B_{k}\times  X}  \S{\psi_k} + \sum_{B_{k}}m u\QQ[u] |\psi_{k}|^{p}\right)\\
		 &\le 
		  2^{-kp}C  \left(\sum_{B_{k}\times X}  \S{\psi}+\sum_{B_{k}}m u\QQ[u] |\psi|^{p} \right).
\end{align*}
{Note that the second inequality uses the fact that $ \sol{u} $ is a positive supersolution of \eqref{e:Q'}}. This finishes the proof for $2\le p<\infty $.\medskip

\emph{Case $ 1< p <2 $}: 
The proof is similar  but a bit more involved than the corresponding proof for $ p\ge2 $. We replace the function $ \psi_{k}=0\vee\left[ (2^{-k+1}\psi-1)\wedge (4-(2^{-k+1}\psi-1))  \right]\wedge 1 $ in the proof  above by $ 
	\psi_{k}^{2/p}. $
Clearly, since $0\le \psi_{k} \leq 1 $, we have
\begin{align*}
	\psi_{k}^{2/p} \leq  \psi_{k} \leq  2^{-k+1} \psi1_{B_{k}}. 
\end{align*}	
Furthermore, we show below
\begin{align*}
	\S{	\psi_{k}^{2/p}} \leq 2^{-kp}C\S{\psi},
\end{align*}
where the constant $ C $ depends only on $ p $.  With these two adjustments one follows the proof  for the case $ p\ge 2 $ verbatim to obtain the desired result. 

Hence, we are left to show $ \S{	\psi_{k}^{2/p}} \leq 2^{-kp}C\S{\psi} $. 	We start by showing a claim.\smallskip
	
	\emph{Claim.} For fixed $ U\ge 0 $, let $ \Psi $ for $  s,t\ge 0 $ be given as
	\begin{align*}
		\Psi(s,t)=\frac{|s-t|^{2}}{(U(s+t ) +|s-t|)^{2-p}} \, 
	\end{align*}
and $ \Psi (0,0)=0 $. 
	Then, for $ a>0 $, $s,t \in[0,a] $ and $ b\ge0 $
	\begin{align*}
		\Psi (s^{2/p},t^{2/p})\leq  C_{p,a,b} \Psi(s+b,t+b),
	\end{align*}
	where $ C_{p,a,b}= 4(2a+2b)^{2-p}/p^{2} $.
	
	\emph{Proof of the  claim.} We assume one of $s,t>0$. Otherwise, the claim follows trivially. By symmetry we can assume $ s\ge t $. 	One easily sees by multiplying out
		$$(s-t)(s^{2/p}+t^{2/p}) \leq (s^{2/p}-t^{2/p}) (s+t) ,$$
		from which we deduce  for all $ b\geq 0 $ 
		$$(s-t)\left[ \frac{s^{2/p}+t^{2/p}}{s+t+2b}\right]\leq (s^{2/p}-t^{2/p}).$$
		Moreover, by the mean value theorem applied to the function $r \mapsto r^{2/p}$ one gets
		$$(s^{2/p}-t^{2/p})^2 \leq( 4/p^2) s^{\frac{2}{p}(2-p)}(s-t)^2.$$
		These two inequalities combined give
	\begin{align*}
	\Psi(s^{2/p},t^{2/p})	&=\frac{(s^{2/p}-t^{2/p})^{2}} {\left(U (s^{2/p}+t^{2/p}) +(s^{2/p}-t^{2/p})\right)^{2-p}} \leq    \frac{ ({4}/{p^2}) s^{\frac{2}{p}(2-p)}(s-t)^{2}}{\left(U(s^{2/p}+t^{2/p}) +\frac{(s^{2/p}+t^{2/p})}{(s+t+2b)}(s-t)\right)^{2-p}} \\
		&= \frac{4}{p^2} s^{\frac{2}{p}(2-p)}\frac{(s+t+2b)^{2-p}}{(s^{2/p}+t^{2/p})^{2-p}} \left(\frac{(s-t)^{2}}{\left(U(s+t+2b) +(s-t)\right)^{2-p}}\right) 	\leq C \Psi (s+b,t+b),
	\end{align*}
where we used $ s^{\frac{2}{p}(2-p)}(s^{2/p}+t^{2/p})^{p-2}\leq 1 $ and $ s+t+2b\leq 2a+2b $.
This proves the claim.\medskip

We now come to the main part of the proof, i.e. to show $ \S{	\psi_{k}^{2/p}} \leq 2^{-kp}C\S{\psi} $. 
	We  let $$  \ph =0\vee  (2^{-k+1}\psi)\wedge 5   \quad\mbox{and}\quad \eta= 0\vee[(\ph-1)\wedge (4-(\ph-1))]\wedge 1. $$
	Observe that $ \eta $ takes values in $ [0,1] $ and satisfies {$\eta=\psi_{k}$.}\smallskip
	
	We  distinguish three cases to show $$  \S{\eta^{2/p} }\leq{C}\S{\ph} .$$ Let $ x,y\in X $ and     $U= U_{xy}= |\nabla_{xy} \sol{u}|/[2(\sol{u}(x) \sol{u}(y))^{{\frac{1}{2}}}] $.\smallskip
	
	\emph{Case 1.} The statement is trivial if   $ \eta(x)=\eta(y)  $.\smallskip
	
	\emph{Case 2.} Assume  $ \eta(x)=0 $ and $ \eta(y)\neq 0 $, and note that this includes $ \eta(x)\neq0 $, $ \eta(y)=0 $ by symmetry.
	 	
	Then $  |\nabla_{xy} \eta ^{2/p}|=\eta(y)^{2/p}= \langle{\eta^{2/p}
	}\rangle_{xy} . $ Moreover, by the case distinction of $ \ph(x)\le 1 $ and 
$ \ph(x)\ge 5 $ (which is exactly when $ \eta(x) =0$), we obtain  $ \eta(y)\le |\nabla_{xy}\ph| $. Furthermore, since  $  \langle{\ph}\rangle_{xy}/5 ,|\nabla_{xy} \ph|/5\le 1 $, we obtain
	\begin{align*}
		\frac{\S{\eta ^{2/p}}_{xy} }{b(x,y)(\sol{u}(x)\sol{u}(y))^\frac{p}{2} }
		&= \frac{|\eta(y)|^{2}}{\left(U_{xy}+1\right)^{2-p}}\leq C\frac{|\nabla_{xy}\ph|^{2}}{\left(U_{xy} \langle{\ph}\rangle_{xy}+|\nabla_{xy}\ph|\right)^{2-p}}
		=C\frac{\S{\ph}_{xy} }{b(x,y)(\sol{u}(x)\sol{u}(y))^\frac{p}{2} }\, .
	\end{align*}

	\emph{Case 3.} Assume  $ \eta(x),\eta(y)\neq 0 $.  We distinguish three further cases which can be all treated by applying the claim proven above.\smallskip
	
	\emph{Case 3.1.} Assume $ \eta(x)=\ph(x)-1 $ and $ \eta(y)=\ph(y)-1 $.
	Then $\ph(x),\ph(y)\in  [1,2] $ and
	 by the claim applied with $ a=1,$  $b=1 $
	\begin{align*}
	\frac{\S{\eta ^{2/p}}_{xy} }{b(x,y)(\sol{u}(x)\sol{u}(y))^\frac{p}{2} }
	\!=\!\Psi((\ph(x)-1)^{\frac2p},({\ph(y)} - 1)^{\frac2p})\leq C \Psi(\ph(x),\ph(y))
	= C \frac{\S{\ph}_{xy} }{b(x,y)(\sol{u}(x)\sol{u}(y))^\frac{p}{2} }\, .
\end{align*}

\emph{Case 3.2.} Assume $ \eta(x)=4-(\ph(x)-1) $ and $ \eta(y)=(4-(\ph(y)-1)) \wedge 1$.  Then  $\ph(x)\in [4,5]$, $ \ph(y)\in[3,5] $ and we assumed without loss of generality $ \ph(x)\ge \ph(y) $.
For such $x,y$, we further set $ \tilde \eta =4-(\ph-1) $ and observe $  \eta^{2/p} =\tilde\eta^{2/p}\wedge 1 $. Thus, by Lemma~\ref{lem:contr}  we have $ \S{\eta ^{2/p}}\leq \S{\tilde\eta ^{2/p} }$.
 We obtain by the claim applied with $ a=1, b=5 $ and since $ |\nabla_{xy}\tilde\eta|= |\nabla_{xy} \ph|$
as well as $  \langle \tilde \eta\rangle_{xy}+5\geq 5\geq \langle \ph \rangle_{xy}$, we have 
\begin{multline*}
	\frac{\S{\eta ^{2/p}}_{xy} }{b(x,y)(\sol{u}(x)\sol{u}(y))^\frac{p}{2} }\leq 	\frac{\S{\tilde\eta ^{2/p}}_{xy} }{b(x,y)(\sol{u}(x)\sol{u}(y))^\frac{p}{2} }
=	\Psi(\tilde\eta(x)^{2/p},\tilde\eta(y)^{2/p})\leq C \Psi(\tilde\eta(x)+5,\tilde\eta(y)+5) 
\\	=C\frac{|\nabla_{xy}\tilde\eta|^{2}}{\left(U_{xy}(\langle \tilde\eta\rangle_{xy}+5)+{|\nabla_{xy}\tilde\eta|}\right)^{2-p}}
	\leq C \frac{|\nabla_{xy}\ph|^{2}}{\left(U_{xy}\langle \ph\rangle_{xy}+{|\nabla_{xy}\ph|}\right)^{2-p}}=	{C}\frac{{\S{\ph}_{xy}} }{b(x,y)(\sol{u}(x)\sol{u}(y))^\frac{p}{2} }\, .
\end{multline*}

\emph{Case 3.3.} Assume $ \eta(x)=(4-(\ph(x)-1))\wedge 1 $ and $ \eta(y)=(\ph(y)-1)\wedge 1 $ such that $ \eta(x)\neq \eta(y) $.  Then we consider either $\ph(x)\in [3,5]$ and $\ph(y)\in  [1,2] $  or $ \ph(x)\in [4,5] $ and $ \ph(y)\in [1,3] $ (as otherwise $ \eta(x)=\eta(y)=1 $ which is trivial and treated in Case 1).
Hence, $  |\nabla_{xy}\ph |\ge 1 $, and by the crude estimates $ 0\leq \eta(x),\eta(y)\leq  1$ as well as {$ 1\leq \ph(y)\leq 3\leq \ph(x)\leq 5 $}, we get
\begin{align*}
	|\nabla_{xy}\eta|\leq 1\leq  |\nabla_{xy}\ph |\quad \mbox{and}\quad 
	\frac{\langle \eta\rangle_{xy}+1}{	|\nabla_{xy}\eta|}\ge 1\ge\frac{1}{4}\cdot 
	\frac{\langle \ph\rangle_{xy}}{|\nabla_{xy}\ph|} \,.
\end{align*}
We proceed to  use these estimates together with  the claim  with $ a=4 $, $ b=1$ to deduce 
\begin{multline*}
\frac{\S{\eta ^{2/p}}_{xy} }{b(x,y)(\sol{u}(x)\sol{u}(y))^\frac{p}{2} }
=	\Psi(\eta(x)^{2/p},\eta(y)^{2/p})\leq C {\Psi(\eta(x)+1,\eta(y)+1)}\\
=C{|\nabla_{xy}\eta|^{p}}{\left(U\frac{\langle \eta\rangle_{xy}+1}{|\nabla_{xy}\eta|}+1\right)^{p-2}}
	\leq C{|\nabla_{xy}\ph|^{p}}{{\left(U\frac{\langle \ph\rangle_{xy}}{|\nabla_{xy}\ph|}+1\right)^{p-2}}}=C\frac{\S{\ph}_{xy} }{b(x,y)(\sol{u}(x)\sol{u}(y))^\frac{p}{2} }\, .
\end{multline*}

With  the estimate $ \S{\eta^{{2/p}}}\leq C\S{\ph}$ at hand and $ \psi_{k}=\eta $,  we use
	Lemma~\ref{lem:contr} for  $\ph= 0\vee  (2^{-k+1}\psi)\wedge 5 $  to conclude
	$$\S{\psi_{k}^{2/p}} {=} \S{\eta^{2/p}}\leq C	\S{\ph} \leq C	\S{2^{-k+1}\psi} =2^{-(k+1)p}C\S{\psi} $$ 
 which finishes the proof.
\end{proof}

\begin{remark} \rm
$(i)$ Observe that our characterization of Hardy weights in Theorem \ref{Thm:Char} characterizes the graphs where the Poincar\'e inequality holds. Indeed, in a graph $b$ on $(X,m)$ with potential $c$, the Poincar\'e inequality holds if and only if there exists $C>0$ such that
$$\sum_{K} m u^p \leq C {\mathrm{Cap}}_{u}(K)$$
for all compact $K \subseteq X$, where $ \sol{u} $ is a strictly positive  supersolution of \eqref{e:Q'}. That is, in such graphs, the measure $m u^p$ is absolutely continuous with respect to the ${\mathrm{Cap}}_{u}$.

$(ii)$ We give some concrete examples of Hardy weights for the $p$-Laplacian $\Delta_p$ on $\mathbb{Z}^d$ using our characterization.  Let $X=\mathbb{Z}^d$ with $d>p$, $b \equiv 1,m \equiv 1$, and $c  \equiv 0$. Then $L^{d/p}(\mathbb{Z}^d)$-functions are Hardy weights for the $p$-Laplacian $\Delta_p$ on $\mathbb{Z}^d$. To show this, first of all recall that $\mathcal{D}_0(\mathbb{Z}^d) \hookrightarrow L^{pd/(d-p)}(\mathbb{Z}^d)$ \cite{Hua}, where $\mathcal{D}_0(\mathbb{Z}^d)$ is as defined in Section \ref{sec_existence}. Using this, for any compact set $K$ in $\mathbb{Z}^d$, we obtain
$$\left(\sum_K 1 \right)^{(d-p)/d} \leq C {\mathrm{Cap}}_{1}(K)$$
for some $C>0$. Now, let $w \in L^{d/p}(\mathbb{Z}^d)$. Then, for any compact set $K$ in $\mathbb{Z}^d$,
\begin{align*}
\frac{\sum_K |w|}{\mathrm{Cap}_{1}(K)} \leq \frac{\left[\sum_K |w|^{d/p}\right]^{p/d} \left[\sum_K 1 \right]^{1-p/d}}{\mathrm{Cap}_{1}(K)} \leq C \|w\|_{L^{d/p}(\mathbb{Z}^d)} \,.
\end{align*}  
Thus, by Theorem \ref{Thm:Char}, we have $L^{d/p}(\mathbb{Z}^d) \hookrightarrow \mathcal{H}_p(\mathbb{Z}^d,1,1,0)$.
\end{remark}
%%%%%%%%%%%%%%%
%%%%%%%%%%%%%%%%%%
\section{Necessary condition for Hardy-weights}\label{KP-condition}

In this section we study necessary conditions for a function to be a Hardy weight. Above we proved a characterization of Hardy weights in terms of a norm involving the $ Q $-capacity and a positive supersolution of \eqref{e:Q'}.  Given a positive solution of minimal growth at infinity we provide a necessary condition for Hardy weights.

	\begin{definition}[Solution of minimal growth at infinity and ground state]
		{ A  function $ u $ is called  a \emph{positive solution of \eqref{e:Q'} of minimal growth at infinity}  in $X$ if $ u>0 $ and $\QQ[u]=0$   in~$X\setminus K_{0}$ for some compact $ K_{0}\Subset X $ and for any  $ v $ such that $ v>0 $ and  $\QQ[v]\geq 0$ in~$X\setminus K$  for some $ K_{0}\Subset K\Subset X $ which satisfies~$u\leq v$ on 
			$K$ one has $u\leq v$ in~$X\setminus K$.  Such a function $ u $ which satisfies $ u>0 $ and $ \QQ[u] =0$ on $ X  $ is called a \emph{global minimal positive solution}}.
	\end{definition} 
 If $\Q$ is critical in $X$, then $Q$ admits an Agmon ground state,  which is in fact a global minimal positive solution, \cite{Florian_cri}.  On the other hand, in the subcritical case, for any $x\in X$ the equation $\QQ[\varphi]=0$ admits a positive solution in $X \setminus \{x\}$ of minimal  growth at infinity in $X$. Such a function is called a {\em minimal positive Green function} of  $\Q$ in $X$ which has a positive charge at $x$, cf. \cite{Florian_cri}. Hence, in either case there is a positive solution  of \eqref{e:Q'} of minimal growth at infinity.

\begin{definition}[Null-sequence] \label{def-gs} 
	A nonnegative sequence $(\varphi_n) $ in $  C_c(X)$
	which satisfies  $ \Q(\ph_{n}) \to {0}$ and $ C^{-1} \leq \ph_{n}(o) \leq C$ for some  
	$ o\in X $ and $ C>0 $,  and for all $ n\in\N $ is called a \emph{null-sequence} for $\Q$ in $X$.
\end{definition}

{
It was shown in   \cite[Theorem 4.1]{Florian_nonlocal} that a nonnegative functional $\Q$ is critical in $X$ if and only if $\Q$ admits a null-sequence in $X$. In this case, any null-sequence converges pointwise to the unique (up to a multiplicative constant) positive solution of the equation $Q'[\varphi]=0$ in $X$, that is to the Agmon ground state.}

 In \cite[Theorem 3.1]{KP} the authors proved, in the continuum case,  a necessary condition for $|g|$ to be a Hardy-weight. Later in \cite{Das_Pinchover1} it is shown that the above result follows directly from the Maz'ya-type characterization proved therein. Similarly, we show here that a corresponding necessary condition on graphs follows directly from Theorem~\ref{Thm:Char}. 
 \begin{theorem}[Necessary condition]\label{Thm:KP} Assume that $Q$ is subcritical in $X$ and let  $ u  $ be a  positive solution of minimal growth at  infinity in $X$.
 	Then
 	\begin{align*}
 		\H \subseteq  \ell^{1}(X,\sol{u}^{p}m ) 
 	\end{align*}
 	  or equivalently if $g \in \H$, then  $  u\in \ell^{p}(X,|g|m) $. 
\end{theorem} 

%%%%%%%%%%%%%%%%%%%%%%%
We need the following two observations.
\begin{lem}\label{lem:crit}
	 Assume that $Q$ is subcritical in $X$. For all $ x \in X$, there is $ c_{0} > 0 $  such that $ \Q_{0}:=\Q-c_{0}1_{x} $ is critical. 
\end{lem}
\begin{proof}
	Since $ Q $ is subcritical, there is a strictly positive Hardy-weight $ w>0 $ of $ Q $, cf. \cite[Corollary 5.6]{Florian_cri}. Hence, $\Q-w(x)1_{x}\geq 0 $. Thus, we can take $c_0>0$ as the largest constant $c$ such that  $\Q-c1_{x} \geq 0 $. If $\Q_{0}=\Q-c_{0}1_{x}$ is subcritical, then again there is a strictly positive Hardy-weight $w$ such that $ \Q_{0} -w$ is nonnegative, and therefore,  $\Q_{0}-w(x)1_{x}\geq 0 $,  contradicting  the definition of $c_0$. 
\end{proof}
\begin{lemma}\label{lem:null}
	If $ \Q $ is critical with an Agmon ground state $ \psi $, then,
	for any $ K \Subset X$ and $ \varepsilon>0 $, there is $ \ph\in C_{c}(X) $ with $1_{K}\psi\leq  \ph\leq \psi $	 such that $ Q(\ph) \leq \varepsilon$.
\end{lemma}
\begin{proof}
	For a critical $ \Q $, there is  $ \psi_{n} \in C_{c}(X)$ such that $ \Q(\psi_{n})\to 0 $ and $ \psi_{n}\to \psi $ pointwise,  \cite[Theorem 4.1]{Florian_nonlocal}. Thus, for all $ \delta>0 $, we have $ \psi_{n}\ge (1-\delta)\psi $ for large $ n $ {in $K$}. Hence, 
	\begin{align*}
		{\mathrm{Cap}_{\psi}(K)} \leq  \frac{1}{(1-\delta)^p} Q(\psi_{n}) \to 0
	\end{align*}
	for $ n\to\infty $.
	By Lemma~\ref{lem:cap}, we infer $ \widetilde { \mathrm{Cap}}_{\psi}(K)=\inf \{\Q(\ph) \mid   \ph \in C_{c}(X),\; 1_K \psi \leq \ph\leq \psi\}=0 $ and the statement follows.
\end{proof}

\begin{proof}[Proof of Theorem~\ref{Thm:KP}] Let $g \in \H$. Let $ K_{n} $ be an exhaustion of $ X $, i.e. $ K_{n}\Subset K_{n+1}\Subset X $ and $ \bigcup_{n}K_{n}= X $. Let $ x\in X $ and $ c_{0} $ as in Lemma~\ref{lem:crit} above such that $ \Q_{0}=\Q-c_{0}1_{x} $ is critical. Let $ \psi>0  $ be an Agmon ground state of $\Q_{0}$. 
By Lemma~\ref{lem:null}  above there are $ \psi_{n}\in C_{c}(X) $ with $1_{K_{n}}\psi\leq  \psi_{n}\leq \psi $	 such that $ Q_{0}(\psi_{n}) \leq 1/n$. Then, with $ C(g)=\|g\|_{\H}^{-1} $
\begin{align*}
	C(g)\sum_{K_{n}}m|g||\psi|^{p}\leq C(g)	\sum_{X}m|g||\psi_{n}|^{p}\leq Q(\psi_{n}) = \Q_{0}(\psi_{n}) + c_{0}|\psi_{n}(x)|^{p}\leq \frac{1}{n}+ c_{0}|\psi(x)|^{p}.
\end{align*}
Taking the limit $ n\to\infty $, we infer $ g\in \ell^{1}(X,\sol{\psi}^{p}m )   $. Let $ u $ be  positive solutions of minimal growth at  infinity in $X$. Let  $ K $ be a compact set containing $x$ such that $u$ is a positive solution in $X\setminus K$.   Take  $ C>0 $ such that $ u\le C \psi  $ on $ K $. {Since $\psi$ is also a positive solution} outside of $ K $, it follows that  $ u\leq C\psi  $ on $ X\setminus K $ and the statement follows.
\end{proof}

\section{Existence of minimizer}\label{sec_existence}
We provide a sufficient condition on $g$ and $c$ so that the best constant in $ C|g|\le Q $ is attained in a certain function space, i.e. there exists a minimizer. So, the first question is to introduce a space of functions in which we look for a minimizer as $ C_{c}(X) $ is clearly too small.

For a non-positive potential $ c $, the energy functional $Q$ is not a norm and not even a semi-norm, so we cannot take the closure of $ C_{c}(X) $ with respect to $ \Q $. Thus, we consider the energy functional with respect to the positive part of the potential  
$ \Q_{+}= \Q_{p,b,c_+} $ where $ c_{+}=c\vee 0 $. Clearly, we can extend $ \Q_{+} $ to
\begin{align*}
\mathcal{D}(X):= \{f \in C(X)\mid \Q_{+}(f)< \infty\}
\end{align*}
in an obvious way on which $  \Q_{+}^{1/p} $ is a semi-norm. In general, we cannot expect to extend $ \Q $ to $ \mathcal{D}(X) $ for non-positive $ c $, so $ \mathcal{D}(X) $ is too large. However, as we will prove below, that it is possible to extend  a subcritical $ \Q $ to the following subspace of $ \mathcal{D}(X) $: Denote by $ \mathcal{D}_{0}(X) $ the space of all functions $ f\in \mathcal{D}(X) $ for which there is $ (\ph_{n}) $ in $ C_{c}(X) $ such that $ \ph_{n}\to f $ with respect to  $  \Q_{+}^{1/p} $ and pointwise.

In order to extend  the energy functional $ \Q $ and  also the Hardy inequalities to $ \mathcal{D}_{0}(X) $, we first show some basic properties of $ \mathcal{D}_{0}(X)  $ in the proposition below.

\begin{proposition}[Properties of $ \mathcal{D}_{0}(X) $] \label{Prop:weak_pointwise}
	Assume $ \Q $ is subcritical. Then  $ \mathcal{D}_{0}(X) $  is an reflexive Banach space with respect to the norm $ \Q_{+}^{1/p} $, and 
	\begin{align*}
		\mathcal{D}_{0}(X)= \overline{C_{c}(X)}^{\Q_{+}^{1/p}}.
	\end{align*}
	Furthermore, weak convergence with respect to  $ \Q_{+}^{1/p} $ in $ \mathcal{D}_{0}(X) $ implies pointwise convergence. 
\end{proposition}
\begin{proof}
 Since $\Q $ is subcritical,  there is $0< g\in \H $ by  \cite[Corollary~5.6]{Florian_cri} such that $ g\leq \Q $ on $ C_{c}(X) $.
Thus, on $ C_{c}(X) $
\begin{align*}
	\Q_{+} \leq \Q_{+}+g \leq \Q_{+}+ \Q\le  2\Q_{+}.
\end{align*}
Hence, $  Q_{+}^{1/p}$ and   $ (Q_{+}+g)^{1/p} $ are equivalent as semi-norms. Due to $ g>0 $, we can infer that   $ (Q_{+}+g)^{1/p} $ is a norm, and convergence with respect to $ (Q_{+}+g)^{1/p}  $ implies pointwise convergence. Hence, these properties carry over to $  Q_{+}^{1/p}$. We  observe that the closure of $ C_{c}(X) $  under  $ (Q_{+}+g)^{1/p} $  is a subspace of $ \mathcal{D}(X)\cap \ell^{p}(X,m|g|) $. Then, by the equivalence of the norms $  Q_{+}^{1/p}$ and   $ (Q_{+}+g)^{1/p} $ and the definition of $ \mathcal{D}_{0}(X) $ 
\begin{align*}
 \mathcal{D}_{0}(X) \supseteq  \overline{C_{c}(X)}^{(\Q_{+}+g)^{1/p}}= \overline{C_{c}(X)}^{\Q_{+}^{1/p}}	\supseteq \mathcal{D}_{0}(X).
\end{align*}
This means that $ \mathcal{D}_{0}(X) $ is the closure of $ C_{c}(X) $ with respect to $  Q_{+}^{1/p}$.

To see that  $ \mathcal{D}_{0}(X) $ is a  reflexive Banach space, observe that the map 
	\begin{align*}
		\mathcal{D}_{0}(X) \ra \ell^p(X\times X, b) \times \ell^p(X,c_{+}),\qquad \ph\mapsto(\nabla\varphi,\varphi)
	\end{align*}
	is an isometry and, therefore, $ \mathcal{D}_{0}(X) $ as a {closed subspace},  inherits the reflexivity of the Banach space $\ell^p(X\times X, b) \times \ell^p(X,c_{+})$.  
	
Finally, we show that weak convergence with respect to $ Q_{+}^{1/p} $ implies pointwise convergence. Recall that weak convergence means convergence under linear functionals which are continuous with respect to $ Q_{+}^{1/p} $. For $ x\in X $ and $ g>0 $ with $ g\leq \Q $,	the linear functional  $$  \Psi_{x}: \mathcal{D}_{0}(X)\to \R,\quad \ph\mapsto m(x)g(x)\ph(x)  $$ is  continuous in  $ \mathcal{D}_{0}(X) $: Indeed, by H\"older inequality with $ \frac{1}{p}+\frac{1}{q}=1$, we obtain
	\begin{align*}
		\Psi_{x}(\ph)=\!\sum_{X}mg1_{x}\ph \leq \left(\sum_{X}mg|1_{x}|^{q} \right)^{\frac{1}{q}}
		\left(\sum_{X}mg|\ph|^{p} \right)^{\frac{1}{p}}  \leq \Q(1_{x})^{{\frac{1}{q}}}\Q(\ph)^{\frac{1}{p}}\le  
		\Q_{+}(1_{x})^{{\frac{1}{q}}}\Q_{+}(\ph)^{\frac{1}{p}}.
	\end{align*}	
So, weak convergence of a sequence $ (f_{n}) $ to $ f $ in $ \mathcal{D}_{0} $ with respect to $ Q^{1/p}_{+} $, in particular implies convergence $  \Psi_{x}(f_n) \to \Psi_{x}(f) $  as $ \Psi_{x} $ is  a continuous functional for every $ x\in X $. Since $ mg>0 $, this implies pointwise convergence $ f_{n}\to f $.
\end{proof}

With these facts at hand, we can extend 
 the Hardy inequalities $ C|g|\le \Q $ for $ g\in\mathcal{H} $ from $ C_{c} (X)$ to $ \mathcal{D}_{0}(X) $. In particular, this means that we can extend $ \Q $ to $ \mathcal{D}_{0}(X) $. 
 %\Hmm{We do not show that $ Q(f_{n}) \to Q(f)$ for $ Q_{+} %$-Cauchy sequences. To this end we would need that $ \sum %c_-|f_{n}|^{p}\to \sum c_{-}|f|^{p} $. However, I do not see why %we have more than Fatous lemma. Fortunately, we do not seem to %need more than the Hardy inequality.}

\begin{proposition}[Extension of Hardy inequality to $ \mathcal{D}_{0}(X) $]\label{p:HardyOnD0}	Assume {that} $ \Q $  is subcritical. If $ f\in \mathcal{D}_{0}(X) $, then $ \sum_{X}c|f|^{p} $ converges absolutely and 
\begin{align*}	\Q(f):=\frac{1}{2}\displaystyle \sum_{X\times X}  b|\nabla f|^p + \sum_{ X} c|f|^p
\end{align*}	
is finite. Furthermore, for
 $ g\in \H $ and $ C>0 $ with $C |g|\leq \Q$ on $ C_{c}(X) $,   we have
\begin{align*}
	C\sum_{X}m|g||f|^{p}\leq \Q(f),\qquad f\in \mathcal{D}_{0}(X).
\end{align*}
\end{proposition}
\begin{proof} 
By $C |g|\leq \Q$ on $ C_{c}(X) $, we have 
		\begin{align*}
		\sum_{X}m(C|g|+c_{-})|\ph|^{p}\leq \Q_{+}(\ph),\qquad \ph\in C_{c}(X).
	\end{align*}
Hence, by definition of $ \mathcal{D}_{0}(X) $, Fatou's Lemma and the proposition above this inequality extends to $ \mathcal{D}_{0}(X) $ which yields the statement as $ \Q=\Q_{+}-c_{-} $.
\end{proof}
%%%%%%%%%%%%%%%%%%%%%%%%
We introduce another semi-norm which can {be  considered} as the $ \|\cdot\|_{\H} $-norm at the ``boundary'' of $ X $, though
 we do not specify it as an object itself. For $ g\in \H $, let
$$  \|g\|_{\H,\partial X} =\inf_{K\Subset X} \sup_{\ph\in C_{c}(X\setminus K),\Q(\ph)\neq 0} \frac{\sum_{X}m|\ph|^{p}|g|}{\Q(\ph)}.  $$
{If $|g|>0$,} then $$  {\lambda_0(g)}:=\frac{1}{\|g\|_{\H}} $$ has a spectral meaning as it can be interpreted as the infimum of $ \Q $ over $ C_{c}(X) $ functions normalized in $ \ell^{p} (X,|g|m)$. Hence, $ \|g\|_{\H} ^{-1}$ can be related to the  bottom of the spectrum of some operator in $ \ell^{p} (X,|g|m)$ which we do {not} specify because we will not use this operator  in what follows. Similarly, 
$$\lambda_{\infty}(g)=  \frac{1}{\|g\|_{\H,\partial X}}  $$ can be seen to be related to the   bottom of the essential spectrum  in $ \ell^{p} (X,|g|m)$.

Hence, we say that $ g $ admits a \emph{spectral gap} if  $ \lambda_{0}(g)<  \lambda_{\infty}(g)  $.
Below we show that under an integrability assumption on $c_-$, a spectral gap  implies that $ \|g\|_{\H} $ is attained in $ \mathcal{D}_{0}(X) $.

\begin{theorem}[Existence of minimizer] \label{Thm:best_constant} 
	Let $g \in \H$ be such that $  \lambda_{0}(g)<  \lambda_{\infty}(g)   $. Assume that for some  $K\Subset X$ there exists  a positive solution $u$ of $ (\QQ-\lambda_{0}(g)|g|)[\varphi]=0 $ in $X\setminus K$ such  that  $  c_{-}\in \ell^{1}(X,|u|^p) $. Then,  $ \|g\|_{\H} $ is attained in $\D_{0}(X)$, i.e. there is $ f\in \D_{0}(X) $ such that $ \|g\|_{\H}\Q(f)=\sum_{X}m|g| |f|^{p} $.
\end{theorem}

We begin  with a key lemma, claiming that if the  above spectral gap condition  is satisfied for a Hardy-weight $g$, then $\Q-\lambda_{0}(g){|g|}$ is critical.  The proof is similar to the proof of \cite[Lemma~2.3]{Lamberti_Pinchover}, {see also the references therein}. %We denote the closure  of a subset $ A\subseteq \R $ by $ \mathrm{cl}(A) $.
%%%%%%%%%%%%%%%%%%%%%
\begin{lemma}[Spectral gap implies criticality]\label{lem-spectral-gap} 
	Let $g \in \H$ be such that $ \lambda_{0}(g)<  \lambda_{\infty}(g)$.   Then  $\Q-\lambda_{0}(g){|g|}$ is critical in $X$. 
\end{lemma}
\begin{proof}
We set
	\begin{align*}
	S &:=\{t\in {\mathbb{R}}\mid (\Q-t|g|)\geq 0 \mbox{ on } C_{c}(X) \}, \\
	S_{\infty} &:=\{  t\in {\mathbb{R}} \mid (\Q -t|g|)\geq 0 \mbox{ on }  C_{c} (X \setminus K) \mbox{ for some compact set } K  \subset X  \}.
	\end{align*}
	Clearly, $S$ and $S_\infty$ are  intervals and since $g$ admits  a spectral gap, it follows that
	$$S=\ (\!-\infty ,\lambda_{0}(g)] \ \varsubsetneq S_{\infty} \subseteq  \ (-\infty ,  \lambda_{\infty}(g) ],  \mbox{ and } 0<\lambda_{0}(g)<\lambda_{\infty}(g) .$$
	To simplify notation, we set $\lambda_0=\lambda_0(g)$. 
	
	Let $ \lambda_{1}\in S_{\infty}\setminus S $. Then there is a compact set $K\neq \emptyset$ such that  $ (\Q-\lambda_{1}|g| )\ge 0 $ on $ C_{c}(X\setminus K) $.  Let 
	$$  w=  (\|(c_{-}/m)1_{K}\|_{\infty}+\lambda_{1}\|g1_{K}\|_{\infty}+1)1_{K}$$
which is compactly supported in $ K $ and nontrivial. We observe that for $ \ph\in C_{c}(X) $ 
	\begin{multline*}
		(\Q-\lambda_{1}|g|+w)(\ph)= (\Q-\lambda_{1}|g|)(\ph 1_{X\setminus K}) +\frac{1}{2}\sum_{K\times K}b|\nabla \ph|^{p}+\sum_{K\times X\setminus K}b|\nabla \ph|^{p}+\sum_{K}c_{+}|\ph |^{p}\\+\sum_{K} m(w-\frac{c_{-}}{m}-\lambda_{1}|g|)|\ph|^{p}\ge 0.
	\end{multline*} 
	Consequently,   $(Q-\gl|g|+w	 )\geq 0$ on $C_{c}(X)$  for all  $ \lambda \in  [\lambda_{0},\lambda_{1}]  $.

	For $ \lambda \in  [\lambda_{0},\lambda_{1}]  $, let
	$$  s(\lambda) := \inf\{ s\in\R\mid (Q-\gl|g|+ sw	 )\geq 0 \mbox{ on  }C_{c}(X)\} . $$
	By the discussion above,   $ s(\lambda)\leq 1$ for $\lambda\in [\lambda_{0},\lambda_{1}]$.
	
	If $ s(\lambda)\leq 0 $, then 
	\begin{align*}
		(Q-\gl|g|	)\geq (Q-\gl|g|+ s(\lambda)w	 )\geq 0\quad \mbox{ on  }C_{c}(X).
	\end{align*}
	Hence,  $ \lambda\in S $ and since  $S=(-\infty,\lambda_{0}] $ as well as $ \lambda\ge \lambda_{0} $ we infer $ \lambda=\lambda_0 $. We conclude $ s(\lambda)>0 $ for $ \lambda>\lambda_0 $.
	
	Furthermore,
	$$\{(\lambda,s)\in [\lambda_0,\lambda_1]\times \R \mid (Q-\gl|g|+sw)\geq 0 \mbox{ in } X\}$$ 
	is obviously convex 	  
	which implies that the function $ s $ is convex. Since, $ s(\lambda)\in (0,1] $ for $ \lambda>\lambda_0 $ as shown above and $  s(\lambda_{0})\le 0$  since $ \Q -\lambda_{0}|g|\ge 0 $, convexity of $ s $ implies $$  s(\lambda_{0})=0 . $$
	
	Now, we use this to show criticality of $  \Q -\lambda_{0}|g|$.  	{We employ \cite[Corollary~11.5]{F_Diss} with $ V=X $,  $ h=Q-\lambda_{0} |g|+w $ which is clearly subcritical and  $ \tilde w =- w $ which is clearly negative somewhere. According to  \cite[Corollary~11.5]{F_Diss}, $Q-\lambda |g|+s(\gl)w$
	is critical for $\lambda\in [\lambda_0,\lambda_1]$. In particular, since $  s(\lambda_{0})=0$, we obtain that 
	$Q-\lambda_{0} |g|$ is critical.}
\end{proof}

\begin{proposition} \label{Prop:null_seq}
	Let $(\varphi_n) \in  C_c(X)$ be a null-sequence for a critical $\Q$ with an Agmon ground state $ \psi $. Then, $(0\vee{\varphi}_n\wedge \psi)$ is also a null-sequence for $\Q$.
\end{proposition}
\begin{proof}
	This follows immediately from Lemma~\ref{l:cuttoff}.
\end{proof}

%%%%%%%%%%%%%%%%%%%%%%%%%%%%%%
\begin{proof}[Proof of Theorem~\ref{Thm:best_constant}] Denote $ \lambda_{0}=\lambda_{0}(g) $.
	Lemma~\ref{lem-spectral-gap} implies that  $Q-\lambda_{0}|g|$ is critical and let $\psi$ be a ground state satisfying $\psi(o)=1$ for some fixed $o\in X$. By the necessary condition Theorem~\ref{Thm:KP}, we have $ \psi\in \ell^{p}(X,|g|m) $. Furthermore, since $ \psi $ is a global minimal positive solution, we infer $ \psi \leq C u $ for some constant $ C $. By our assumption $  u\in \ell^{p}(X,c_{-})  $,  therefore, $$  \psi\in \ell^{p}(X,|g|m+c_{-}) . $$
	By Proposition~\ref{Prop:null_seq}, $ (\Q-{\lambda_0}|g|) $ admits a null-sequence $ (\psi_{n}) $ such that $ 0\le\psi_{n}\leq\psi  $. Moreover, as already discussed after the definition of null-sequences, $ \psi_{n} \to\psi$ pointwise by  \cite[Theorem~4.1]{Florian_nonlocal}. Thus,
	using $ \psi_{n}\leq \psi $
	\begin{align*}
	\Q_{+}(\psi_{n}) \leq  (\Q-{\lambda_0}|g|)({\psi}_{n})+\sum_{X}(m{\lambda_{0}}|g|+c_{-})|\psi|^{p}
	\end{align*}
 and as $ (\psi_{n}) $ is a $ (\Q-{\lambda_0}|g|) $-null-sequence, the above estimate also gives that $ \Q_{+}(\psi_{n}) $ is bounded. As $ \mathcal{D}_{0}(X) $ is a reflexive Banach space, we infer that up to a subsequence $ (\psi_{n}) $ converges weakly in $ \mathcal{D}_{0} (X) $ by the Alaoglu-Bourbaki theorem. By Proposition~\ref{Prop:weak_pointwise} we infer pointwise convergence and, hence, the weak limit agrees with $ \psi $ which we already discussed to be the pointwise limit. In particular, $ \psi \in \mathcal{D}_{0}(X) $. Furthermore,  from the estimate above, the fact that $(\psi_{n}) $ is a $ (\Q-{\lambda_0}|g|) $-null-sequence and Fatou's lemma, we obtain
 \begin{align*}
 	\Q_{+}(\psi) \leq  \sum_{X}(m{\lambda_{0}}|g|+c_{-})|\psi|^{p}
 \end{align*}
Reordering this inequality and combining it with the Hardy inequality for $ g $ which extends to $ \mathcal{D}_{0}(X) $ by Proposition~\ref{p:HardyOnD0}, we obtain
	\begin{align*}
		\lambda_{0}\sum_{X}m|g|{|\psi|^p}\leq 	\Q(\psi)\leq 	\lambda_{0}\sum_{X}m|g|{|\psi|^p}
	\end{align*}
which yields the statement.
\end{proof}
%%%%%%%%%%%%%%%

\section{Cheeger constant}\label{sec_Cheeger}

In this section we relate the  best Hardy constant to Cheeger constants. As discussed in the previous section the reciprocal of the best Hardy constant can be understood as a spectral quantity. Hence, it does not come as a surprise that such a spectral quantity can be estimated in terms of a Cheeger constant.

For a symmetric weight $ a:X\times X\to[0,\infty)  $ over $ (X,m) $, we define the corresponding \emph{Cheeger constant} by
\begin{align*}
	h({a,m}):=\inf_{W\Subset X} \frac{a(\partial W)}{m(W)}
\end{align*}
where  the \emph{boundary} $ \partial W $ of $ W $ is given by
\begin{align*}
	\partial W = W\times (X\setminus W) %\cup (X\setminus W)\times W.
\end{align*}
and {$ a(\partial W)=\sum_{\partial W}a $}.

Furthermore, for $ p\in (1,\infty) $ let $ q $ be its H\"older conjugate, i.e. $ 1/p+1/q=1 $. A  pseudo-metric $\gr$ is called $ p\, $-intrinsic for $ b $ over $ (X,m) $ if
\begin{align*}
	\frac{1}{m(x)}\sum_{y\in X}b(x,y)\rho(x,y)^{q}\leq 1, \quad x\in X.
\end{align*}
It is referred to as an \emph{intrinsic metric} in the case $ p=2 $, cf. \cite{Bauer,FLW,KLW}.

Below we will provide two theorems with corresponding  corollaries alluding to special cases where the best Hardy constant is estimated by the reciprocal of Cheeger constant.

The first case is concerned with general $ p \in {(1,\infty)}$ but restricted to vanishing potential. This case applies the results of  \cite{KellerMugnolo}. Secondly, to omit the assumption on vanishing of the potential, one can restrict oneself to $ p=2 $. This case applies the results of \cite{Bauer}.
%%%%%%%%%%%%
\begin{theorem}[$ p $-Cheeger estimate] Assume $ p\in (1,\infty) $ and $ c=0 $. Let $ g:X\to (0,\infty)$   and let  $ \rho  $ be $ p\, $-intrinsic for $ b $ over $ (X,|g|m) $. Then,
	\begin{align*}
		\frac{1}{h({b,m|g|})}\leq 	    \|g\|_{\H}\leq \frac{2^{1-p}p^{p}}{h({b\rho,m|g|})^{p}}\, ,
	\end{align*}
where the quotients are infinite if the denominator is zero.
\end{theorem}
\begin{proof}
	The upper bound follows directly from the definition of $  \|g\|_{\H} $ and \cite[Theorem~3.1]{KellerMugnolo}. The lower bound then follows from  \cite[Theorem~3.9.]{KellerMugnolo} employed with respect to the combinatorial graph distance.
\end{proof}

For a function $ g:X\to (0,\infty) $, denote
\begin{align*}
	D:=	\sup_{x\in X}	\frac{1}{m(x)|g|(x)}\sum_{y\in X}b(x,y)
\end{align*}
which takes values in $(0,\infty]$. We obtain  the following corollary if $ D $ is bounded.

\begin{corollary}\label{c:Cheeger_p} Assume $ p\in (1,\infty) $ and $ c=0 $. For all $ g:X\to (0,\infty) $  such that $ D<\infty $, we have
	\begin{align*}
		\frac{1}{h({b,m|g|})}\leq 	    \|g\|_{\H}\leq \frac{2^{1-p}p^{p}D^{p/q}}{h({b,m|g|})^{p}}.
	\end{align*}
\end{corollary}
\begin{proof}
Observe that the combinatorial graph distance $ d $ divided by $ D^{1/q} $ is a $ p $-intrinsic metric. Moreover, 
$ h({bd/D^{1/q},m|g|})^{p} =  h({b,m|g|})^{p}D^{-p/q} $ since $ bd=b $. Thus, the statement follows directly from the theorem above.
\end{proof}

We now turn the case of arbitrary potentials $ c $ satisfying   $ \Q_{2,b,c}\ge 0 $ on $ C_{c}(X) $.

\begin{theorem}[Cheeger estimate for ground state transform] Let $ p=2 $, let  $ u $ be a strictly positive harmonic function and $ g:X\to (0,\infty) $.
Then, for an intrinsic metric $ \rho  $ for $ b_{u} =b (u\otimes u) $ over $ (X,m|g|u^{2}) $, i.e.
\begin{align*}
	\frac{1}{m(x)|g(x)|}\sum_{y\in X}\frac{u(y)}{u(x)}b(x,y)\rho(x,y)^{2}\leq 1, \quad x\in X,
\end{align*}
we have
\begin{align*}
	\frac{1}{h({b_{u},m|g|u^{2}})}\leq 	    \|g\|_{\H}\leq \frac{2}{h({b_{u}\rho,m|g|u^{2}})^{2}}.
\end{align*}	
\end{theorem}
\begin{proof}
	First observe that by the ground state transform \cite[Proposition 4.8 ]{KPP} (or \cite[Theorem~4.11]{KLW}) we have for all nontrivial $ \ph\in C_{c}(X) $
	\begin{align*}
		\frac{Q_{2,b,c}(\ph)}{\sum_{X}m|g|\ph^{2}} =
		\frac{Q_{2,b_{u}}(\ph/u)}{\sum_{X}mu^{2}|g|(\ph/u)^{2}} .
	\end{align*}
	Thus, the upper bound follows directly from the definition of $  \|g\|_{\H} $ and \cite[Theorem~3.1.]{Bauer}. The lower bound then follows from  \cite[Theorem~3.6.]{Bauer} employed with respect to the combinatorial graph distance.
\end{proof}

We say a function $ u:X\to (0,\infty) $  is of \emph{bounded oscillation} if
\begin{align*}
	U:=\sup_{x\sim y}\frac{u(y)}{u(x)}<\infty
\end{align*}

\begin{corollary}Let $ p=2 $, let  $ u $ be a strictly positive harmonic function  of bounded oscillation and $ g:X\to(0,\infty) $. Then, for  an intrinsic metric $ \rho  $ for $ b $ over $ (X,m|g|) $, we have
\begin{align*}
	\frac{1}{h({b_{u},m|g|u^{2}})}\leq 	    \|g\|_{\H}\leq \frac{2U}{h({b_{u}\rho,m|g|u^{2}})^{2}}.
\end{align*}	
\end{corollary}
\begin{proof}
From the condition on an intrinsic metric it is clear, that an intrinsic metric for $ b $ over $ (X,m|g|) $ divided by $ U^{1/2} $  is intrinsic for $ b_{u} $ over $ (X,m|g|u^{2})  $. Thus, the statement follows from the theorem above.
\end{proof}

\begin{corollary}Let $ p=2 $, let  $ u $ be a strictly positive harmonic function  of bounded oscillation, $ g:X\to(0,\infty) $ and $ D<\infty $. Then,  have
	\begin{align*}
		\frac{1}{h({b_{u},m|g|u^{2}})}\leq 	    \|g\|_{\H}\leq \frac{2UD}{h({b_{u},m|g|u^{2}})^{2}}.
	\end{align*}	
\end{corollary}
\begin{proof}
Let $d$ be the combinatorial graph distance. As discussed in the proof of Corollary~\ref{c:Cheeger_p}, $ d/D^{1/2} $ is intrinsic for $ b $ over $ (X,m|g|) $. Hence, the statement follows from the corollary above.
\end{proof}

%%%%%%%%%%%%%%%%%%

%%%%%%%%%%%%%%%
\begin{center}
	{\bf Acknowledgments}
\end{center}
U.D. and Y.P.  acknowledge the support of the Israel Science Foundation (grant 637/19) founded by the
Israel Academy of Sciences and Humanities. U.D. is also supported in part by a fellowship from the Lady Davis Foundation. M.K. is grateful for the generous hospitality at the Technion and the financial support by the Swiss Fellowship during the time this research was conducted. Additionally, {the authors acknowledge} the support of the DFG supporting a stay of U.D. at the University of Potsdam and M.K. at the Technion.
%%%%%%%%%%%%%%%%%%%%%

%%%%%%%%%%%%%%%%%%%%%%%%
%\bibliography{Ref}
%\bibliographystyle{abbrv}
\end{document}